\documentclass[11pt,a4paper,leqno]{amsart}
\usepackage{amsfonts,amssymb,amsmath,amsthm,graphicx}
\usepackage{mathrsfs}
\usepackage[abbrev]{amsrefs}
\usepackage{amscd}
\usepackage{xspace}
\usepackage{verbatim}

\newtheorem{theorem}{Theorem}[section]
\newtheorem{proposition}[theorem]{Proposition}
\newtheorem{corollary}[theorem]{Corollary}
\newtheorem{lemma}[theorem]{Lemma}
\theoremstyle{definition}
\newtheorem{definition}[theorem]{Definition}
\theoremstyle{remark}
\newtheorem{remark}[theorem]{Remark}
\numberwithin{equation}{section}

\newcommand{\R}{{\mathbb R}}

\renewcommand{\L}{\mathscr{L}}
\newcommand{\K}{\mathbb{K}}
\newcommand{\G}{\mathbb{G}}
\newcommand{\M}{\mathfrak{M}}
\newcommand{\tr}{\mathrm{tr}^*}
\newcommand{\loc}{\mathrm{loc}}
\newcommand{\trm}{\mathrm{tr}}

\newcommand{\eps}{{\varepsilon}}

\newcommand{\beqn}{\begin{eqnarray}}
\newcommand{\eeqn}{\end{eqnarray}}   
\newcommand{\beq}{\begin{eqnarray*}}
\newcommand{\eeq}{\end{eqnarray*}}

\newcommand{\usup}{\overline{u}}
\newcommand{\usub}{\underline{u}}

\newcommand{\BA}{\begin{array}}
\newcommand{\EA}{\end{array}}
\newcommand{\BAN}{\renewcommand{\arraystretch}{1.2}
\setlength{\arraycolsep}{2pt}\begin{array}}
\newcommand{\BAV}[2]{\renewcommand{\arraystretch}{#1}
\setlength{\arraycolsep}{#2}\begin{array}}

\newcommand{\BSA}{\begin{subarray}}
\newcommand{\ESA}{\end{subarray}}
\newcommand{\BAL}{\begin{aligned}}
\newcommand{\EAL}{\end{aligned}}

\newcommand{\note}[1]{\noindent\textit{#1.}\hspace{2mm}}
\newcommand{\Remark}{\note{Remark}}
\newcommand{\forevery}{\quad \forall}

\newcommand{\norm}[1]{\left \|#1\right \|}

\newcommand{\rec}[1]{\frac{1}{#1}}

\newcommand{\dist}{\mathrm{dist}\,}

\newcommand{\prt}{\partial}
\newcommand{\sms}{\setminus}

\newcommand{\ti}{\times}

\newcommand{\sbs}{\subset}

\newcommand{\ind}[1]{_{_{#1}}}
\newcommand{\chr}[1]{\mathbf{1}\ind{#1}}

\newcommand{\sth}{such that\xspace}

\newcommand{\bdw}{\partial\Gw}

\def\bcom{}
\def\ga{\alpha}     \def\gb{\beta}       \def\gg{\gamma}
       \def\gd{\delta}      \def\ge{\epsilon}

    \def\gr{\rho}        
       
      \def\gw{\omega}
\def\gx{\xi}                
\def\Gg{\Gamma}     \def\Gd{\Delta}      

\def\Gl{\Lambda}    \def\Gs{\Sigma}      
\def\Gw{\Omega}              

      \def\CC{{\mathcal C}}

\def\BBG {\mathbb G}

\def\BBP {\mathbb P}   \def\BBR {\mathbb R}

\def\GTM {\mathfrak M}

\def\BHP{boundary Harnack principle\xspace}

\begin{document}

\title[Semilinear elliptic equations with Hardy potential]{Moderate solutions of semilinear elliptic equations with Hardy potential under minimal restrictions on the potential}

\author{Moshe Marcus}
\address{Department of Mathematics\\ Technion\\ Haifa 32000\\ ISRAEL}
\email{marcusm@math.technion.ac.il}

\author{Vitaly Moroz}
\address{Department of Mathematics\\ Swansea University\\ Singleton Park\\
Swansea SA2~8PP\\ Wales, UK}
\email{v.moroz@swansea.ac.uk}



\begin{abstract}
We study semilinear elliptic equations with Hardy potential
$$-\L_\mu u+u^q=0\leqno{(E)}$$ in a bounded smooth domain $\Omega\subset \R^N$. Here $q>1$, $\L_\mu=\Delta+\frac{\mu}{\delta_\Omega^2}$ and $\delta_\Omega(x)=\dist(x,\partial\Omega)$. Assuming that $0\leq \mu<C_H(\Omega)$, boundary value problems with measure data and discrete boundary singularities for positive solutions of $(E)$ have been studied in \cite{MMN}. In the case of \emph{convex} domains $C_H(\Gw)=1/4$. In this case similar problems have been studied in \cite{Veron}. In the present paper we study these problems, in \emph{arbitrary domains}, assuming only $-\infty<\mu<1/4$, even if $C_H(\Omega)<1/4$. We recall that $C_H(\Omega)\leq 1/4$ and, in general, strict inequality holds. The key to our study is the fact that, if $\mu<1/4$ then in  smooth domains there exist local $\L_\mu$-superharmonic functions in a neighborhood of $\partial\Omega$ (even if $C_H(\Omega)<1/4$). Using this fact we extend the notion of \emph{normalized boundary trace} introduced in \cite{MMN}, to arbitrary domains, provided that $\mu<1/4$. Further we study  the b.v.p.\ with normalized boundary trace $\nu$ in the space of positive finite measures on $\partial\Omega$. We show that existence depends on two critical values of the exponent $q$ and discuss the question of uniqueness.
Part of the paper is devoted to the study of the linear operator: properties of local $\L_\mu$-subharmonic and superharmonic functions and the related notion of moderate solutions. Here we extend and/or improve results of \cite{BMR08} and  \cite{MMN} which are later used in the study of the nonlinear problem.
\end{abstract}

\maketitle

\section{Introduction and main results} \label{Intro}

\subsection{Introduction.}
On bounded smooth domains $\Omega\subset\R^N$ ($N\geq 2$) we study semilinear elliptic equations with Hardy potential of the form,
\begin{equation*}\tag{$P_{\mu}$}\label{*-intro}
-\Delta u-\frac{\mu}{\delta_\Omega^2}u+|u|^{q-1}u=0\quad\mbox{in }\Omega,
\end{equation*}
where $q>1$, $-\infty<\mu<1/4$ and
$$\delta_\Omega(x):=\dist(x,\partial\Omega).$$
Equations $(P_0)$ had been extensively studied in the past two decades and by now the structure of the set of positive solution of such equations is well understood, see \cite{Marcus-Veron} and further references therein. Equation \eqref{*-intro} with Hardy potential, i.e. with $\mu\neq 0$, had been first considered in \cite{BMR08}, where a classification of positive solutions had been introduced and conditions for the existence and nonexistence of {\em large} solutions for \eqref{*-intro} had been derived.

The study and classification of positive solutions of equation \eqref{*-intro} relies on the properties of the associated linear equation
\begin{equation}\label{e:0}
-\L_\mu h=0\quad\mbox{in }\:\Omega,
\end{equation}
where
$$\L_\mu:=\Delta+\frac{\mu}{\delta_\Omega^2}.$$
Denote
$$\alpha_\pm:=\frac{1}{2}\pm\sqrt{\frac{1}{4}-\mu}$$
and note that $\alpha_++\alpha_-=1$.
For $\rho>0$ and $\eps\in(0,\rho)$ we use the notation
$$
\begin{array}{ll}
\Omega_\rho:=\{x\in\Omega:\delta(x)<\rho\}, & \Omega_{\eps,\rho}:=\{x\in\Omega:\eps<\delta(x)<\rho\} \vspace{\jot}\\
D_\rho:=\{x\in \Omega: \delta(x)>\rho\},
& \Sigma_\rho:=\{x\in\Omega:\delta(x)=\rho\}.
\end{array}
$$

A function $w\in L^1_{\loc}(G)$ is a $\L_\mu$-subharmonic in $\Omega$ if $\L_\mu w\leq 0$  in the distribution sense, i.e.,
\begin{equation*}
\int_{G}w(-\Delta \varphi)\,dx-
\int_{G}\frac{\mu}{\delta_\Omega^2}w\varphi\,dx\leq \:0
\qquad\forall\:0\le\varphi\in C^\infty_c(\Omega).
\end{equation*}
We say that $w$ is a {\em local} $\L_\mu$-subharmonic function if there exists $\rho>0$ such that
$w\in L^1_{\loc}(\Omega_\rho)$ is subharmonic in
$\Omega_\rho$. Similarly, (local) $\L_\mu$-superharmonic functions are defined with ``$\geq$'' in the above inequality.

\subsection{The role of the Hardy constant.}
The existence and properties of positive $\L_\mu$-harmonic and superharmonic functions in $\Omega$ are controlled by the Hardy constant of the domain, defined as
\begin{equation}\label{e-Hardy}
C_H(\Omega):=\inf_{C^\infty_c(\Omega)\setminus\{0\}}
\frac{\int_{\Omega}|\nabla u|^2\,dx}{\int_{\Omega}\frac{u^2}{\delta_\Omega^2}\,dx}.
\end{equation}

For a bounded  Lipschitz domain it is known that $C_H(\Omega)\in (0,1/4]$.
If $\Omega$ is convex then $C_H(\Omega)=1/4$. In general, $C_H(\Omega)$ varies with the domain
and could be arbitrary small (see, e.g. \cite{MMP}*{Theorem I and Section 4}) for a discussion and examples).

Denote the {\em local} Hardy constant in $\Omega_\rho$ relative to $\bdw$ by
\begin{equation}\label{e-localHardy}
C_H^{\bdw}(\Omega_\rho):=\inf_{C^\infty_c(\Omega_\rho)\setminus\{0\}}
\frac{\int_{\Omega_\rho}|\nabla u|^2\,dx}{\int_{\Omega_\rho}\frac{u^2}{\delta_{\Omega}^2}\,dx}.
\end{equation}
Note the difference between  $C_H^{\bdw}(\Gw_\rho)$ and $C_H(\Gw_\rho)$: the distance involved in the  first one is $\gd_\Gw(x)=\dist(x,\bdw)$ while in the second it is $\gd_{\Gw_\rho}(x)=\dist(x,\bdw_\rho)$. Obviously $C_H^{\bdw}(\Omega_\rho)\geq C_H(\Omega_\rho)$.

The following lemma shows that in contrast to the "global" Hardy constant $C_H(\Omega)$
the value of the "local" Hardy constant $C_H^{\bdw}(\Omega_\rho)$ does not depend on the shape of $\Omega$,
provided that $\rho$ is sufficiently small.

\begin{lemma}\label{t-Hardy-loc}
{\sc (Local Hardy Inequality)}
There exists $\bar\rho=\bar\rho(\Omega)>0$ such that for every $\rho\in(0,\bar\rho]$ one has $C_H^{\bdw}(\Omega_\rho)=C_H(\Omega_\rho)=1/4$.
\end{lemma}

The fact that $C_H^{\bdw}(\Omega_\rho)=1/4$ is due to \cite{MMP}*{p.3246}, while $C_H(\Omega_\rho)=1/4$ follows from \cite{Brezis-Marcus}*{Lemma 1.2}.

The relation between Hardy constant and the existence of positive $\L_\mu$-superharmonics
is explained by the following classical result, cf. \cite{MMP}*{p.3246}.
\begin{lemma}\label{t-GST}
Equation \eqref{e:0} admits a positive $\L_\mu$-superharmonic function in $\Omega$ if and only if $\mu\le C_H(\Omega)$.

Equation \eqref{e:0} admits a positive $\L_\mu$-superharmonic in $\Omega_\rho$ with $\rho\in(0,\bar\rho)$ if and only if $\mu\le 1/4$.

\end{lemma}

Thus, according to Lemma \ref{t-Hardy-loc}, if $C_H(\Omega)<1/4$ then, for $\mu\in[C_H(\Omega),1/4)$, there exist local positive $\L_\mu$-superharmonic functions but no  ``global'' positive $\L_\mu$-superharmonic functions in $\Omega$.

\subsection{Moderate solutions and normalised boundary trace.}
In this work we study {\em moderate} positive solutions of nonlinear equation \eqref{*-intro} in the range $\mu<1/4$,
including negative values of $\mu$. Recall that in the classical theory of equations \eqref{*-intro} with $\mu=0$, moderate solution is a solution which is dominated by a positive harmonic function, cf. \cite{Marcus-Veron}*{pp.66-69}. This concept had been extended to equations \eqref{*-intro}
with $0\le\mu<C_H(\Omega)$ in \cite{MMN}, where $\L_\mu$-moderate solution is defined as a solution dominated by a positive $\L_\mu$-harmonic function.
This definition is not applicable in the range $\mu\in[C_H(\Omega),1/4)$, when the set of positive $\L_\mu$-harmonic function is empty.
Therefore we modify it as follows:

\begin{definition}\label{d-moderate}
A solution $u\in L^1_{loc}(\Omega)$ of equation \eqref{*-intro} is $\L_\mu$--{\em moderate} if there exists a local positive $\L_\mu$-harmonic function $h$ such that
$|u|\le h$ in $\Omega_\rho$ for some $\rho\in(0,\bar\rho]$.
\end{definition}

We are going to show that equation \eqref{*-intro} admits $\L_\mu$-moderate solutions, with prescribed (normalized) boundary data, in the {\em entire} domain $\Omega$
for every $\mu<1/4$, even when $C_H(\Omega)<1/4$.
The \emph{existence} of a certain class of positive solutions was observed in \cite{BMR08}*{Lemma 4.15}.

More specifically, we study the generalised boundary trace problem
\begin{equation*}\tag{$P_{\mu}^{\nu}$}\label{*nu}
\left\{
\begin{array}{rcll}
-\L_\mu u+|u|^{q-1}u&=&0&\text{in $\Gw$},\smallskip\\
\tr_{\bdw}(u)&=&\nu,&
\end{array}
\right.
\end{equation*}
where $\mu< 1/4$, $q>1$, $\nu\in \mathcal M^+(\partial\Omega)$ and $\tr_{\bdw}(u)$
denotes the {\em normalized boundary trace} of a positive Borel function $u$ on $\bdw$. A function $u\in L^q_{loc}(\Gw)$ is a solution of \eqref{*nu} if it satisfies the equation in the distribution sense and attains the indicated boundary data.

The concept of  normalised  boundary trace was introduced in \cite{MMN} in order to classify positive moderate solutions of \eqref{*nu} in terms of their behaviour at the boundary,  when $0<\mu<C_H(\Omega)$.\footnote{Actually, the assumption $\mu>0$ was introduced in \cite{MMN} only for simplicity: the normalised  boundary trace is well-defined and  the related results remain valid for any $\mu<C_H(\Gw)$.} It is defined as follows.

A nonnegative Borel function $u:\Omega\to\R$ possesses a {\em normalised boundary trace} $\nu\in\M^+(\partial\Omega)$ if,

\begin{equation}\label{ntr-intro}
\lim_{\eps\to 0}\frac{1}{\eps^{\alpha_-}}\int_{\Sigma_\eps}\big|u-\K_\mu^{\Omega}[\nu]\big|dS=0
\end{equation}
where $K_\mu^\Gw$ is the Martin kernel of $\L_\mu$ in $\Omega$. If, for a given $u$ there exists a measure $\nu$ as above then it is unique.

By Ancona \cite{Ancona-87}, if $\mu<C_H(\Gw)$ there is a (1-1) correspondence between the set of positive $\L_\mu$-harmonic functions in $\Gw$ and $\M^+(\partial\Omega)$; the $\L_\mu$-harmonic function $v$ corresponding to a measure $\nu$ has the representation  $v=K_\mu^\Gw[\nu]$. (For details and notation see Subsection 2.1 below.)

 We point out that, except in the case $\mu=0$,  $\mathrm{tr}^*_{\bdw}(u)$ is \emph{not} the standard measure boundary trace of $u$. In fact, when $\mu> 0$, the measure boundary trace of any $\L_\mu$-harmonic function is zero.

In order to extend the definition of normalised boundary trace to arbitrary $\mu<1/4$ we pick $\rho\in (0,\bar\rho]$ (with $\bar\rho$ as in Lemma \ref{t-Hardy-loc})  and employ \eqref{ntr-intro} with $K_\mu^{\Gw_\rho}$ instead of $K_\mu^\Gw$. Since $C_H(\Gw_\rho)=1/4$,  $K_\mu^{\Gw_\rho}$ is well defined
for every $\mu<1/4$.

We show that if, for some $\rho$ as above, there exists $\nu\in \M_+(\bdw)$ such that
\begin{equation}\label{ntr-rho}
\lim_{\eps\to 0}\frac{1}{\eps^{\alpha_-}}\int_{\Sigma_\eps}\big|u-\K_\mu^{\Omega_\rho}[\nu]\big|dS=0
\end{equation}
then \eqref{ntr-rho} holds
for every $\rho\in (0,\bar\rho]$ and the measure $\nu$ is independent of $\rho$.

In addition we show that a positive solution of equation \eqref{*-intro} possesses a normalised boundary trace if and only if it is a moderate solution.

\subsection{Main results.} We start with a few results about the linear operator.

\begin{theorem}\label{subh1} Let $\mu<1/4$. Suppose that $u$ is positive and $\L_\mu$-subharmonic in $\Gw_{\bar\rho}$. Then $u$ has a normalized boundary trace on $\bdw$ if and only if  $u$ is dominated in $\Gw_\rho$ (for some $\rho\in (0,\bar\rho)$) by an $\L_\mu$-harmonic function.
\end{theorem}

\begin{theorem}\label{subh2}  Let $\mu<1/4$. Suppose that $u$ is a non-negative, $\L_\mu$-subhar\-mo\-nic function  in $\Gw_{\bar\rho}$. In addition assume that, for some $\rho\in(0,\bar\rho)$ $u$ is dominated in $\Gw_\rho$ by an $\L_\mu$-harmonic function.
Then, either

(i) $\tr_{\bdw}u=0$, in which case, for every $\gb\in (0,\rho)$ there exists a constant $c_\gb>0$ such that
\begin{equation}\label{trace*=0}
u(x)\leq c_\gb \gd(x)^{\ga_+} \quad\text{in }\Gw_\gb
\end{equation}
or

(ii) $\tr_{\bdw}u>0$, in which case, for every $\gb$ as above,
\begin{equation}\label{trace*>0}
\rec{c_\gb} \gb(x)^{\ga_-}\leq \int_{\Gs_\gb}udS\leq c_\gb\gb^{\ga_-} \quad\text{in }\Gw_\gb.
 \end{equation}
\end{theorem}

\begin{theorem}\label{superh1}
Let $\mu<1/4$. Suppose that $u$ is positive and $\L_\mu$-super\-har\-monic in $\Gw_{\bar\rho}$.
If $\tr_{\bdw}u\neq 0$ then \eqref{trace*>0} holds.
 \end{theorem}

\begin{corollary}\label{c-subh}
Suppose that $u$ is non-negative and $\L_\mu$-subharmonic in $\Gw_{\bar\rho}$. Then either \eqref{trace*=0} holds or
\begin{equation}\label{e-subh}
0<\limsup_{\gb\to 0}\rec{\gb^{\ga_-}}\int_{\Gs_\gb}udS.
\end{equation}
\end{corollary}

\begin{remark}
The corollary is an improved version of \cite{BMR08}*{Thm. 2.9}. Since we do not assume that $u$ is dominated by an $\L_\mu$-harmonic function the alternative to \eqref{trace*=0} is not necessarily \eqref{trace*>0} but only \eqref{e-subh} which is nothing more than the negation of the statement $\tr_{\bdw} u=0$.
\end{remark}

Clearly every positive subsolution  of the nonlinear equation \eqref{*-intro} is $\L_\mu$-subharmonic so that the above results apply to it.

We turn to the nonlinear problem.

\begin{theorem}\label{Gw-ex-intro}
Let $\mu< 1/4$ and $\nu\in\M^+(\bdw)\setminus\{0\}$. Assume that $\K_\mu^{\Omega_\rho}[\nu]\in L^q_{\delta^{\alpha_+}}(\Omega_\rho)$ for some $\rho\in(0,\bar\rho]$.
Then the boundary value problem \eqref{*nu} admits a positive solution $u$.
\end{theorem}

We emphasise that if $C_H(\Omega)<1/4$ then for $\mu\in[C_H(\Omega),1/4)$ an $\L_\mu$-harmonic extension of $\nu$ exists only locally in a strip $\Omega_\rho$. Nevertheless,  problem \eqref{*nu}
has a positive solution  in  $\Omega$, for any $\mu<1/4$ .

When $\mu<C_H(\Gw)$ problem \eqref{*nu} admits at most one solution for every $\nu\in \M_+(\bdw)$ \cite{MMN}. However,  if $C_H(\Gw)<\mu<1/4$ then uniqueness fails. Indeed,
it was proved in \cite{BMR08}*{Theorem 5.3} that in the latter case there exists a positive solution of \eqref{*nu} with $\nu=0$.
An alternative, more direct proof, of this result is presented in Appendix \ref{App-B}.

\begin{theorem}\label{limit-intro} Let $u$ be a positive solution of \eqref{*-intro}. Then,

(i) $u$ has a normalized boundary trace if and only if $u\in L^q(\Gw;\gd^{\ga_+})$.

(ii) If $u$ has normalized boundary trace $\nu$ then
\begin{equation}\label{e-nontang-moderate-intro}
\lim_{x\to y}\frac{u(x)}{\K_\mu^{\Omega_\rho}[\nu](x)}=1
\quad\text{non-tangentially, for $\nu$-a.e. $y\in\partial\Omega$}.
\end{equation}
\end{theorem}

\begin{theorem}\label{Knu-Lq} Let $\nu\in \M_+(\bdw)$. If $\K_\mu^{\Omega_\rho}[\nu]\in L^q(\Gw;\gd^{\ga_+})$ then \eqref{*nu} has a solution.
\end{theorem}
In general,  the existence of a solution of \eqref{*nu} does not imply that $\K_\mu^{\Omega_\rho}[\nu]\in L^q(\Gw;\gd^{\ga_+})$. In fact, for any $\mu>0$ and $q>1$, one can construct functions $f\in L^1(\bdw)$ such that $\K_\mu^{\Omega_\rho}[f]\not\in L^q(\Gw;\gd^{\ga_+})$ while \eqref{*nu} has a solution whenever $\nu=f\in L^1(\bdw)$.

Let
\begin{equation}\label{tlq-intro}
q_{\mu,c}:=
\frac{N+\alpha_+}{N-1-\alpha_-}\forevery\mu<1/4.
\end{equation}

The next  result has been obtained in \cite{MMN}*{Theorems E and F} for $\mu\in(0,C_H(\Omega)$.  A similar result is presented in \cite{Veron}*{Theorems D and E}, under the assumption that $\Gw$ is a convex domain, in which case it is known that $C_H(\Gw)=1/4$.

\begin{proposition}
Let $\mu< 1/4$.
If $1<q<q_{\mu,c}$ then the boundary value problem \eqref{*nu} has a solution for every Borel measure $\nu\in\M^+(\bdw)$. Moreover, if $q\geq q_{\mu,c}$
then problem \eqref{*nu} has no solution when $\nu$ is the Dirac measure.
\end{proposition}

In the next proposition, the existence statement is a consequence of Theorem \ref{Knu-Lq}. The non-existence part is more subtle.

\begin{proposition}\label{L1data-intro}
(i) For every $\mu<1/4$ put
$$q^*_\mu=\begin{cases} \infty& \quad\text{if }\mu\geq0\\
1-\frac{2}{\alpha_-} & \quad\text{if }\mu<0.
\end{cases}$$

If $1<q< q^*_\mu$ then problem \eqref{*nu} has a solution for every measure $\nu=fdS$, $f\in L^1(\bdw)$.

(ii) If $q\geq q^*_\mu$ then problem \eqref{*nu} has no solution for any $\nu\in \M_+(\bdw)\sms \{0\}$.
\end{proposition}

\begin{remark}
If $\mu<0$ then $\ga_-<0$ so that $q^{*}_\mu>1$ and $q_{\mu,c}<q^*_\mu$.
\end{remark}
\medskip

The  paper is organised as follows. In Section \ref{sect-2} we study the linear problem. We derive estimates of the Green and Martin kernels of $\L_\mu$ in $\Gw_\rho$
and discuss the boundary behavior of local positive $\L_\mu$-sub and superharmonic functions in terms of the normalized trace.

In Section \ref{sect-3} these results are applied to the study of the nonlinear boundary value problem \eqref{*nu}.

\section{Linear equation and normalised boundary trace}\label{sect-2}

\subsection{The local behavior of Green and Martin kernels.} We  recall some results concerning Schr\"odinger equations, that are needed in what follows. The results are due to Ancona \cite{Ancona-87}.
Let $D$ be a bounded Lipschitz domain and consider the Schr\"odinger operator $\L^V=\Gd +V$ where $V\in C(D)$ is a potential \sth, for some constant $a>0$, $|V(x)|\leq a \, \dist(x,\prt D)^{-2}$ and $\L^V$ possesses a positive supersolution. (If $V\le 0$ there is always a supersolution namely, $u=1$.)  Then $\L^V$ has a Green function $G^V$ and Martin kernel $K^V$ in $D$. The Martin boundary coincides with $\prt D$ and the following holds,
\begin{theorem}[Representation Theorem]\label{t-repr}
For every $\nu\in\M^+(\partial D)$ the function
$$\K^V[\nu](x):=\int_{\partial D}K^V(x,y)d\nu(y),\quad x\in D,$$
is $\L^V$-harmonic in $D$.
Conversely, if $u$ is a positive  $\L^V$-harmonic function in $D$ then
there exists a unique measure $\nu\in\M^+(\partial D)$ such that $u=\K^V[\nu]$.
\end{theorem}

In order to state the \BHP we need additional notation.
Let $y\in \prt D$ and let $\gx=\gx^y$ be a local set of coordinates centered at $y$ \sth the $\gx_1$-axis is in the direction of an interior pseudo normal $\mathbf{n_y}$. (If $D$ is a $C^1$ domain we may take $\mathbf{n_y}$ to be the interior unit normal.)
Denote
$$T_y(r,\rho)= \{\gx=(\gx_1,\gx'): |\gx_1|<\rho,\quad |\gx'|<r\}.$$
 Assume that $r$ and $\rho$ are so chosen that
$$\gw_y:=T_y(r,\rho)\cap D=\{\gx: F_y(\gx')<\gx_1<\rho,\, |\gx'|<r\}$$
where $F_y$ is a Lipschitz function in $\BBR^{N-1}$, with Lipschitz constant $\Gl$,  $F_y(0)=0$ and $12\Gl  < \rho/r$. Since $D$ is a bounded Lipschitz domain $\Gl,r,\rho$ can be chosen independently of $y\in
\prt D$.

Let $A\in T(r,\rho)$ be the point \sth  $\gx(A)=(\rho/2,0)$. Then the \BHP reads as follows:
If $u,v$ are positive $\L_\mu$-harmonic functions in $\gw_y$ vanishing continuously on $\bdw\cap T_y(r,\rho)$ then
\begin{equation}\label{BHP}
C^{-1} \frac{u(A)}{v(A)}\leq \frac{u(\gx)}{v(\gx)}\leq C \frac{u(A)}      {v(A)} \forevery \gx\in T_y(r/2,\rho/2)\cap D,
\end{equation}
where the constant $C$ depends only on $N, M,\rho/r$ and the Lipschitz constant of $F_y$, say $\Gl$.  ($\Gl$ may be taken to be independent of $y\in \prt D$.)

We also need the following consequence of the \BHP (c.f.\ Ancona \cite{An-CR82}*{Lemma 3.5}):
there exist positive numbers $c , t_0$ \sth
\begin{equation}\label{K.G}
c^{-1}|x-y|^{2-N}\leq K^V(x,y)G^V(x,x_0)\leq c|x-y|^{2-N}
\end{equation}
for every $y\in \bdw'$ and $x$ on the interior pseudo normal at $y$ \sth $|x-y|\leq t_0$.

Recall that if $V(x)=\mu \dist(x,\prt D)^{-2}$ and $\mu< C_H(D)$ then $\L^V$ has a positive supersolution.
In particular, if $D=\Gw_{\bar\rho}$ then $C_H(D)=1/4$. Therefore, in this case, the above results apply to the operator $\L_\mu=\Delta+\frac{\mu}{\gd_\Gw^2}$ for every $\mu<1/4$.
\medskip

\noindent{\em Notation.} \hskip 2mm
Let $D$ be a subdomain of $\Gw$ and denote
$$\L_{\mu,D}= \Delta+\frac{\mu}{\gd_D^2}\quad \text{where} \quad \gd_D(x)=\dist (x,\prt D).$$
Assume that $\mu<C_H(D)$ and let $D'$ be a subdomain of $D$. Obviously $C_H(D')\geq C_H(D)$. Denote the Green kernel (resp. the Martin kernel) of $\L_\mu$ in $D$ by $G_{\mu}^D$ (resp. $K_\mu^D$). Denote the Green kernel (resp. the Martin kernel) of $\L_{\mu,D}$ in $D'$ by $G_{\mu,D}^{D'}$ (resp. $K_{\mu,D}^{D'}$).

\begin{lemma}\label{KOR1} Assume that $\mu<1/4$. Let $\bar\rho$ be as in Lemma \ref{t-Hardy-loc} and $t\in (0,\bar\rho)$. Put $U = \Gw_{\bar\rho}=[\gd(x)<\bar\rho]$, $\Gw_t=[\gd(x)<t]$, $U_t = [\bar\rho>\gd(x)>t]$.
Then,
\begin{equation}\label{Gr-above}\BAL
&G_\mu^{\Gw_{t/2}}(x,y)\leq \\
&C(t)\inf(|x-y|^{2-N}, \gd(x)^{\ga_+} \gd(y)^{\ga_+}|x-y|^{2\ga_- -N}) \forevery x,y\in \Gw_{t/2}
\EAL\end{equation}
\end{lemma}

\proof
Note that
$\L_\mu=\L_{\mu,U}$ in $\Gw_{t/2}.$ Hence
$$G_{\mu}^{\Gw_{t/2}}=G_{\mu,U}^{\Gw_{t/2}}.$$
 It is well-known that the Green function is monotone with respect to the domain. Therefore $G_{\mu,U}^{\Gw_{t/2}}<G_{\mu,U}^{\Gw_{t}}$ which implies
\begin{equation}\label{Gr1}
 G_{\mu}^{\Gw_{t/2}}(x,y)\leq c G_{\mu,U}^{\Gw_t}(x,y) \forevery x,y\in \Gw_{t/2}.
\end{equation}
By \eqref{Gr1} and the estimate of the Green function of $\L_{\mu,U}$  (see \cite{Tertikas} and \cite{MMN}*{(2.6)}),
\begin{equation}\label{Gr2}\BAL
G_{\mu}^{\Gw_{t/2}}(x,y)&\leq cG_{\mu,U}^{\Gw_t}(x,y)\leq cG_{\mu,U}^U(x,y) \\
&\sim \inf(|x-y|^{2-N}, \gd(x)^{\ga_+} \gd(y)^{\ga_+}|x-y|^{2\ga_- -N})
\EAL\end{equation}
for every $x,y\in \Gw_{t/2}$. This implies \eqref{Gr-above}.
\qed

\begin{theorem}\label{KOR} Assume that $\mu<1/4$, let $\bar\rho$ be as in Lemma \ref{t-Hardy-loc} and let $t\in (0,\bar\rho/2)$. Using the notations of the previous lemma,
pick $x_t\in U_t$ and $x'_t\in \Gw_t$ \sth $\gd(x_t)=(t+\bar\rho)/2$ and $\gd(x'_t)=t/2$.
As usual  $G_0^{U}$ denotes  the Green function for  $-\Delta$ in $U$. A  similar notation is employed for the corresponding Martin kernels.
Then,
\begin{equation}\label{Gest}\BAL
 c_1(t)^{-1}G^U_{\mu,U}(x,x_t)\leq &G^U_\mu(x,x_t)\leq c_1(t)G^U_{\mu,U}(x,x_t) \forevery x\in \Gw_t\\
  c_2(t)^{-1}G_0^U(x,x'_t)\leq &G^U_\mu(x,x'_t)\leq c_2(t)G_0^U(x,x'_t) \forevery x\in U_t,
\EAL\end{equation}
and
\begin{equation}\label{Kest}\BAL
 c_3(t)^{-1}K^U_{\mu,U}(x,y)\leq &K_\mu^U(x,y)\leq c_3(t)K^U_{\mu,U}(x,y) \forevery (x,y)\in \Gw_t\times\bdw,\\
  c_4(t)^{-1}K_0^U(x,y)\leq &K_\mu^U(x,y)\leq c_4(t)K_0^U(x,y) \forevery (x,y)\in U_t\times\Gs_{\bar\rho}.
\EAL\end{equation}
\end{theorem}

\proof
Note that
$\L_\mu=\L_{\mu,U}$ in $\Gw_{\bar\rho/2}$.
Hence both $G_\mu^U(\cdot,x_t)$ and
$G^U_{\mu,U}(\cdot,x_t)$ are $\L_\mu$-harmonic in $\Gw_t$ and vanish on $\bdw$. Therefore, by the \BHP  they are equivalent in a strip $S$  along $\bdw$. In addition they are continuous and bounded away from zero in $\Gw_t\sms S $. This implies the first inequality in \eqref{Gest}. For the second inequality: $G_\mu^U(\cdot,x'_t)$ is $\L_\mu$-harmonic in $U_t$, $G_0^U(\cdot,x'_t)$ is
$\Delta$ harmonic in $U_t$ and $\L_\mu-\Delta= \mu/\gd(x)^2$  is bounded in $U_t$. Therefore, since they both vanish on $\Gs_{\bar\rho}$, we can still apply the \BHP (c.f. Ancona \cite{An-Green97}) to deduce that they are equivalent in the strip $U_t$. This implies the second inequality in \eqref{Gest}.

Recall that, $G^U_{\mu,U}(x,x_t)\sim \gd_{U}(x)^{\ga_+}$ in  $\Gw_t$ for $t\in (0,\rho)$.  (Of course the constants involved in this relation depend on $t$.) Since $\gd_\Gw\sim \gd_U$ in $\Gw_t$, this fact and \eqref{Gest} imply,
 \begin{equation}\label{gdga}
  G_\mu^U(x,x_t)\sim \gd_{\Gw}(x)^{\ga_+} \forevery x\in \Gw_t.
 \end{equation}

In what follows we use the notation introduced for the statement of the \BHP.
Let $y\in \bdw$ and let $\gx=\gx_y$ be a local set of coordinates at $y$ relative to $U$. Thus $$\gw_y=T_y(r,\rho)\cap U= \{\gx: F_y(\gx')<\gx_1<\rho,\, |\gx'|<r\}.$$
We assume that $\gg=\rho/r>12\Gl$.

Since $K_\mu^U(\cdot,y)$ and $G_\mu^U(\cdot,x_t)$ satisfy the (classical) Harnack inequality
\eqref{K.G} remains valid in $\CC_y(b)\cap T_y(r,\rho)$. Therefore, assuming that $\rho<t<\bar\rho$,
\begin{equation}\label{Kest1}
K_\mu^U(\gx,y)G_\mu^U(\gx,x_t)\sim K_\mu^U((\gx_1,0),y)G_\mu^U((\gx_1,0),x_t)\sim |\gx|^{2-N}
\end{equation}
for every $\gx\in \CC_y(b)\cap T_y(r,\rho)$. By \eqref{gdga} and \eqref{Kest1},
\begin{equation}\label{Kest1'}
K_\mu^U(\gx,y)\sim |\gx|^{2-N}\gd(\gx)^{-\ga_+} \forevery \gx\in \CC_y(b)\cap T_y(r,\rho).
\end{equation}

Let $\eta$ be a point in  $\BBR^{N-1}$ \sth $0<|\eta|<r/2$ and denote by $P$ the point  $(F_y(\eta), \eta)$ in the local coordinates $\gx_y$. Then $P\in \bdw$ and  $\gx_P:=\gx_y-P$ is a standard set of local coordinates at $P$. Choose $r_P,\rho_P$ \sth $r_P=|\eta|/2$ and $\rho_P/r_P=\gg$. Then,  $$|x-y|=|\gx_y|\sim |\gx'_y|\sim r_P \forevery x\in \Gw\cap T_P(r_P,\rho_P).$$

Let $A_P=(\rho_P/2,0)$ in $\gx_P$ coordinates, i.e., $A_P=(F_y(\eta)+\gg r_P/2, \eta)$ in $\gx_y$ coordinates.
Pick $b$  such that $\Gl< b< 2\Gl$. Then
$$F_y(\eta)+\rho_P/2\geq -\Gl |\eta| - \gg r_P/2=|\eta|(-\Gl+ \gg/4)>2\Gl |eta|.$$
Consequently,
$F_y(\eta)< b|\eta|< F_y(\eta)+\rho_P/2$, which implies
 $$A_P\in \CC_y(b):=\{\gx_y=(\gx_1,\gx'):\, \gx_1>b |\gx'|\}.$$
Observe that
$$\gd_\Gw(A_P)\sim \gr_P/2, \quad|\gx_y(A_P)|=|A_P-y|\sim (\rho_P^2+r_P^2)^{1/2}\sim r_p. $$
Therefore, by \eqref{Kest1'},
$$K_\mu^U(A_P,y)\sim r_P^{2-N-\ga_+}.$$
In fact, $$|x-y|=|\gx_y|\sim  r_P \forevery x\in \Gw\cap T_P(r_P,\rho_P).$$
Therefore applying \eqref{BHP} in $\Gw\cap T_P(r_P,\rho_P)$ with $u(x)=K_\mu^U(x,y)$  we obtain,
\begin{equation}\label{Kest2}\BAL
K_\mu^U(x,y)&\sim K_\mu^U(A_P,y)\frac{G_\mu^U(x,x_t)}{G_\mu^U(A_P,x_t)}\sim r_P^{2-N-\ga_+}(\gd(x)/r_P)^{\ga_+}\\
&\sim |x-y|^{2-N-2\ga_+}\gd(x)^{\ga_+}=\gd(x)^{\ga_+}|x-y|^{2\ga_- -N}
\EAL\end{equation}
for every $x \in \Gw\cap T_P(r_P/2,\rho_P/2)$.
Combining \eqref{Kest1'} and \eqref{Kest2}, we obtain,
 \begin{equation}\label{Kest2'}
K_\mu^U(x,y)\sim |x-y|^{2-N-\ga_+}(\gd(x)/|x-y|)^{\ga_+}=\gd(x)^{\ga_+}|x-y|^{2\ga_- -N}
\end{equation}
for every $x\in T_y(r/2,\rho/2)$. As \eqref{Kest2'} holds uniformly with respect to $y\in \bdw$ we conclude that there exists $r'>0$ \sth this relation holds
for every $(x,y)\in \Gw_{r'}\times \bdw$. Consequently, for every $t\in (0,\bar\rho)$,
  \begin{equation}\label{Kest3}
K_\mu^U(x,y)\sim |x-y|^{2-N-\ga_+}(\gd(x)/|x-y|)^{\ga_+}=\gd(x)^{\ga_+}|x-y|^{2\ga_- -N}
\end{equation}
for every $(x,y)\in \Gw_{t}\times \bdw$
with similarity constants depending on $t$. Since  $K^U_{\mu,U}$ behaves precisely in the same way (see \cite{MMN}*{Sec. 2.2}) we obtain the first inequality in \eqref{Kest}. The second inequality is proved in a similar way.
\qed
\medskip

We state below two key results concerning the operator $\L_\mu$ in $U=\Gw_{\bar\rho}$. These  have been recently  proved in \cite{MMN}, with respect to the operator $\L_\mu$ in $\Gw$ under the assumption that $0<\mu<C_H(\Gw)$. (In fact, the condition $\mu>0$ is redundant and does not affect the proofs.) Since $C_H(\Gw_{\bar\rho})=1/4$, the results apply to the operator $\L_{\mu,\Gw_{\bar\rho}}$ for every $\mu<1/4$. In view of the relation between the Martin kernels and Green functions of $\L_{\mu,\Gw_{\bar\rho}}$ and $\L_\mu$ in $\Gw_{\bar\rho}$, these results also apply to the operator $\L_\mu$ in $\Gw_{\bar\rho}$.

\begin{theorem}\label{key-res}
(i) If $\nu_0\in\M^+(\partial\Omega)\setminus\{0\}$ then there exist positive numbers $c$ and $\rho_0<\bar\rho$ \sth,
\begin{equation}\label{e-trace-0-}
c^{-1}\norm{\nu_0}\leq\frac{1}{\eps^{\alpha_-}}\int_{\Sigma_\eps}\K_\mu^{\Gw_\rho}[\nu_0]dS \leq c\norm{\nu_0},\quad \ge\in(0,\rho_0).
\end{equation}

\noindent(ii) Let $\rho\in (0,\bar \rho)$ and let $\tau$ be a Radon measure in $\Gw_{\bar\rho}$. Denote
$$\G_\mu^{\Omega_\rho}[\tau](x):=\int_{\Omega_\rho}G_\mu^{\Omega_\rho}(x,y)d\tau(y),\quad x\in\Omega_\rho.$$
If $\tau\in\M^+_{\delta^{\alpha_+}}(\Omega_\rho)$ then for every
$0<\eps<\rho'<\rho$,
\begin{equation}\label{Gmt'}
   \rec{\eps^{\ga_-}} \int_{\Gs_\eps}\G_\mu^{\Gw_\gr}[\tau]dS_x\leq c\int_{\Gw_\rho}\gd^{\ga_+}d\tau,
\end{equation}
where $c$ is a constant depending on $\mu,\rho'$, but not on $\eps$. Moreover,
\begin{equation}\label{e-trace-potential}
\lim_{\eps\to 0}\frac{1}{\eps^{\alpha_-}}\int_{\Sigma_\eps}\G_\mu^{\Omega_\rho}[\tau]\,dS=0.
\end{equation}
\end{theorem}

\begin{remark}
If $\G_\mu^{\Omega_\rho}[\tau](x')<\infty$ for some point $x'\in \Gw_\rho$ then $\tau\in\M^+_{\delta^{\alpha_+}}(\Omega_\rho)$ and $\G_\mu^{\Omega_\rho}[\tau](x)<\infty$ for every $x\in \Gw_\rho$. This follows from the fact that there exists  $c>0$ \sth for every fixed $x\in \Gw_\rho$,
$$\rec{c}\gd(y)^{\ga_+}\leq G_\mu^{\Gw_\rho}(x,y)\leq c\gd(y)^{\ga_+}   \forevery y\in \Gw_{\gd(x)/2}.$$
\end{remark}

\proof
In view of \eqref{Kest3}, inequality \eqref{e-trace-0-} follows from \cite{MMN}*{Corollary 2.11}.

The proof of \eqref{Gmt'} and \eqref{e-trace-potential} is similar to that of \cite{MMN}*{Proposition 2.12}. However several modifications are needed; therefore we provide  the proof of these statements in detail.

We may assume that $\tau>0$. Denote $v:=\BBG_\mu^{\Gw_\rho}[\tau].$ We start with the proof of
\eqref{Gmt'}.

By Fubini's theorem and \eqref{Gest},
\[\BAL
\int_{\Gs_\gb}vdS_x\leq c\Big(&\int_\Gw\int_{\Gs_\gb\cap B_{\frac{\gb}{2}}(y)}|x-y|^{2 -N}
dS_x\,d\tau(y)\\
+\gb^{\ga_+}&\int_\Gw\int_{\Gs_\gb\sms B_{\frac{\gb}{2}}(y)}|x-y|^{2\ga_- -N}
dS_x\,\gd^{\ga_+}(y)d\tau(y)\Big)=I_1(\gb)+I_2(\gb).
\EAL\]
Note that, if $x\in\Gs_\gb$ and $|x-y|\leq \gb/2$ then $\gb/2\leq \gd(y)\leq 3\gb/2$. Therefore
$$\BAL
I_1(\gb)&\leq c_1\gb^{-\ga_+}\int_{\Gs_\gb\cap B_{\frac{\gb}{4}}(y)}|x-y|^{2 -N}dS_x\int_{\Gw_\rho} \gd(y)^{\ga_+}\,d\tau(y)\\
&\leq c'_1\gb^{1-\ga_+}\,\int_{\Gw_\rho} \gd(y)^{\ga_+}\, d\tau(y)
=c'_1\gb^{\ga_-}\,\int_{\Gw_\rho} \gd(y)^{\ga_+}\, d\tau(y)
\EAL$$
and
$$I_2(\gb)\leq c_2\gb^{\ga_+}\int_{\gb/4}^\infty r^{2\ga_- -N}r^{N-2}dr\int_{\Gw_\rho} \gd(y)^{\ga_+}\,d\tau\leq c'_2\gb^{\ga_-}
\int_{\Gw_\rho} \gd(y)^{\ga_+}\,d\tau.$$
This implies \eqref{Gmt'}.

Given $\ell\in (0,\norm{\tau}_{\GTM_{\gd^\ga_+}(\Gw)})$ and $\gb_1\in (0,\gb_0)$ put $\tau_1=\tau\chi_{_{\bar
D_{\gb_1}}}$ and $\tau_2=\tau-\tau_1$. Pick $\gb_1=\gb_1(\ell)$ such that
\begin{equation}\label{tau2}
 \int_{\Gw_{\gb_1}}\gd(y)^{\ga_+}\,d\tau\leq\ell.
\end{equation}
 Thus the choice of $\gb_1$ depends on the rate at which $\int_{\Gw_\gb}\gd^\ga_+\,d\tau $ tends to zero as
 $\gb\to 0$.

Put $v_i=\BBG_\mu^\Gw[\tau_i]$.
Then, for $0<\gb<\gb_1/2$,
$$\int_{\Gs_\gb}v_1\,dS_x\leq c_3\gb^{\ga_+}\gb_1^{2\ga_- -N}\int_{\Gw_\rho} \gd^{\ga_+}(y)d\tau_1(y).$$
Thus,
\begin{equation}\label{Gmt1}
   \lim_{\gb\to 0}\rec{\gb^{\ga_-}} \int_{\Gs_\gb}v_1\,dS_x=0.
\end{equation}
On the other hand, by \eqref{Gmt'} (replacing $\Gw_\rho$ by $\Gw_{\beta_1}$) and \eqref{tau2},
\begin{equation}\label{Gmt2}
   \rec{\gb^{\ga_-}} \int_{\Gs_\gb}v_2\,dS_x\leq c\ell \forevery \gb<\gb_1.
\end{equation}
This proves \eqref{e-trace-potential}.
\qed
\medskip

\begin{corollary}\label{trace=0}   Let $\rho\in (0,\bar\rho]$ and assume that
 $h$ is a nonnegative $\L_\mu$-harmonic function in $\Omega_\rho$ \sth
\begin{equation}\label{e-trace=0}
\lim_{\eps\to 0}\frac{1}{\eps^{\alpha_-}}\int_{\Sigma_\eps}h\, dS=0.
\end{equation}
Then: (i) $h=\K_\mu^{\Omega_\rho}[\nu_\rho]$ for some measure $\nu_\rho\in\M^+(\Sigma_\rho)$ and
(ii) For $t\in (0,\bar\rho)$,
\begin{equation}\label{e-trace-0++}
h\sim \gd_\Gw^{\ga_+} \quad \text{in $\Gw_t$},
\end{equation}
with the similarity constant depending on $t$.
\end{corollary}
\proof
(i) By the Representation Theorem, $h=\K_\mu^{\Gw_\rho}[\nu]$ for some $\nu\in \M(\prt \Gw_\rho)$.  By \eqref{e-trace-0-} and \eqref{e-trace=0},
$\nu_0:=\nu\chr{\bdw}=0$. Thus $\nu=\nu_\rho:=\nu\chr{\Sigma_\rho}$.

\noindent(ii) This is a consequence of (i) and \eqref{Kest3}.
\qed

\begin{corollary}\label{tr(G)}
If  $\tau\in  \M^+_{\gd^{\ga_+}}(\Gw_\rho)\sms\{0\}$   then there exists a positive constant $c=c(\tau)$ \sth
\begin{equation}\label{e-trace-potential-lower}
\G_\mu^{\Omega_\rho}[\tau](x) \geq c \delta(x)^{\alpha_+} \quad \forall x\in \Gw_{\rho},
\end{equation}
and
\begin{equation}\label{e-trace-potential-upper}
\liminf_{x\to \partial \Omega}\frac{\G_\mu^{\Omega_\rho}[\tau](x)}{\gd(x)^{\alpha_-}}< \infty.
\end{equation}

\end{corollary}

\begin{proof} Let $t\in (0,\rho)$ be a number such that $\tau(\Gw_\rho\sms \Gw_t)>0$. Let $\tau'\in \M_+(\Gw_\rho)$ be defined by: $\tau'=\tau$ in $\Gw_\rho\sms \Gw_t$ and $\tau'=0$ in $\Gw_t$. Then  $$\G_\mu^{\Gw_\rho}[\tau]\geq\\G_\mu^{\Gw_\rho}[\tau']:=h.$$
Since $h$ is $\L_\mu$-harmonic in $\Gw_t$, \eqref{e-trace-potential-lower} is a consequence of \eqref{e-trace-0++}

Inequality \eqref{e-trace-potential-upper} follows from \eqref{Gmt'}.
\end{proof}

The next result was proved in \cite{MMN} for $\L_\mu$ in a domain $\Gw$ \sth $\mu<C_H(\Gw)$.

\begin{theorem}\label{tr-sub}
Let $w$ be a nonnegative $\L_\mu$-subharmonic function in $\Omega_\rho$.
If $w$ is dominated by an $\L_\mu$-superharmonic function in $\Omega_\rho$ then
$\L_\mu w=\lambda\in\M^+_{\delta^{\alpha_+}}(\Omega_\rho)$ and there exists
$\nu\in\M^+(\partial\Omega_{\rho})$ such that
\begin{equation}\label{KmGm}
w=\K_\mu^{\Omega_\rho}[\nu]-\G_\mu^{\Omega_\rho}[\lambda].
\end{equation}
\end{theorem}

\proof
There exists a nonnegative Radon measure $\lambda$ in $\Omega_\rho$, such that $-\L_\mu w=-\lambda$ in $\Omega_\rho$.
Since $w$ is dominated by an $\L_\mu$-superharmonic function in $\Omega_\rho$ one shows, as in the proof of \cite{MMN}*{Proposition 2.14}, that
$\lambda\in\M_{\delta^{\alpha_+}}(\Omega_\rho)$. Then $v:=w + \G_\mu^{\Omega_\rho}[\lambda]$
is a nonnegative $\L_\mu$-harmonic function in $\Omega_\rho$. By the Representation Theorem,
$v=\K_\mu^{\Omega_\rho}[\nu]$ for some $\nu\in\M^+(\partial\Omega_\rho)$.
\qed

\begin{definition}\label{traces}
A Borel function $u:\Omega\to\R$ possesses a {\em normalised boundary trace} $\nu_0\in\M^+(\partial\Omega)$ if,
for some $\rho\in(0,\bar\rho]$,
\begin{equation}\label{e-trace}
\lim_{\eps\to 0}\frac{1}{\eps^{\alpha_-}}\int_{\Sigma_\eps}\big|u-\K_\mu^{\Omega_\rho}[\nu_0]\big|dS=0.
\end{equation}
The normalised boundary trace on $\bdw$ will be denoted by $\mathrm{tr}^*_{\bdw}(u)$.
\end{definition}
\Remark Since $u$ is a Borel function $u\lfloor_{\Gs_\rho}$ is well defined and \eqref{e-trace} implies that this function is in $L^1(\Gs_\ge)$ for all sufficiently small $\ge$.

We say that $u$ has a \emph{measure boundary trace} on $\Gs_\rho$ if there exists $\nu_1\in \M^+(\Gs_\rho)$ \sth
$$\lim_{a\to\rho-0}\int_{\Gs_a} u\phi\,dS\to \int_{\Gs_\rho}\phi \,d\nu_1 \forevery \phi\in C_0(\bar\Gw_\rho).$$
This trace is denoted by  $\trm_{\Gs_\rho}(u)$. If both $\trm_{\Gs_\rho}(u)$ and $\tr_{\bdw}(u)$ exist then the measure $\nu\in \M_+(\bdw_\rho)$ given by $\nu\chr{\bdw}=\tr_{\bdw} (u)$ and $\nu\chr{\Gs_\rho}=\trm_{\Gs_\rho}(u)$ is denoted by $\trm^\mu_{\bdw_\rho}(u)$.

\begin{lemma}\label{tr-unique}
The normalised boundary trace $\nu_0$ is uniquely defined, independently of $\rho$.
\end{lemma}

\proof
First we note that \eqref{e-trace} remains valid if $\nu_0$ is replaced by any measure $\nu\in \M_+(\bdw_\rho)$ such that $\nu_0=\nu\chr{\bdw}$. This follows from the fact that, for every measure $\nu_\rho\in \M_+(\Gs_\rho)$,
$$\lim_{\eps\to 0}\frac{1}{\eps^{\alpha_-}}\int_{\Sigma_\eps}\K_\mu^{\Gw_\rho}[\nu_\rho]dS =0.$$
This implies that if \eqref{e-trace} holds with respect to some  $\rho\in(0,\bar\rho)$
then it is valid for any  $\rho'$ in this range. Suppose for instance that $\rho<\rho'<\bar \rho$ and put
$v=\K_\mu^{\Gw_{\rho'}}[\nu_0]$. Let $\nu\in \M_+(\bdw_\rho)$ be the measure equal to $\nu_0$ on $\bdw$ and to $h=v\lfloor_{\Gs_\rho}d\gw_\rho$ on $\Gs_\rho$. (Here $\gw_\gr$ is the $\L_\mu$-harmonic measure on $\Gs_\rho$ relative to $\Gw_{\rho'}$. Since $\Gs_\rho$ is `smooth' $\gw_\rho$ is absolutely continuous with respect to surface measure.) Then $v= \K_\mu^{\Gw_\rho}[\nu]$ in $\Gw_\rho$ and
$$\lim_{\eps\to 0} \frac{1}{\eps^{\alpha_-}}\int_{\Sigma_\eps}|\K_\mu^{\Gw_\rho}[\nu]-\K_\mu^{\Gw_\rho}[\nu_0]|dS=0.$$

It remains to verify that, if \eqref{e-trace} holds then $\nu_0$ is uniquely determined by $u$ in a fixed domain $\Gw_\rho$.

 Suppose, by negation, that there exist  $\nu_1,\nu_2\in \M_+(\bdw)$ \sth \eqref{e-trace} holds for both $v_1=K_\mu^{\Gw_\rho}[\nu_1]$ and $v_2=K_\mu^{\Gw_\rho}[\nu_2]$. Then $w:=|v_1-v_2|$ is $\L_\mu$-subharmonic and $\tr_{\bdw}(w)=0$.

Clearly $w$ is dominated by the $\L_\mu$-superharmonic function $v_1+v_2$.
Therefore, by Theorem \ref{tr-sub}  there exist $\lambda\in\M^+_{\delta^{\alpha_+}}(\Omega_\rho)$ and
$\chi\in\M^+(\partial\Omega_{\rho})$ such that,
$$
  w= \K_\mu^{\Omega_\rho}[\chi]-\G_\mu^{\Omega_\rho}[\lambda].
$$
Thus $w+\G_\mu^{\Omega_\rho}[\lambda]$ is $\L_\mu$-harmonic. By \eqref{e-trace-potential} and the fact that $\tr_{\bdw}w=0$ we have $\tr_{\bdw} (w+\G_\mu^{\Omega_\rho}[\lambda])=0$. Hence $w=0$
 and therefore $\nu_1=\nu_2$.
\qed

\begin{theorem}\label{c-tr-sub} Let $w$ be a nonnegative $\L_\mu$-subharmonic function in $\Omega_\rho$ dominated by an $\L_\mu$-superharmonic function in this domain. Then the boundary trace $\nu=\trm^\mu_{\bdw_\rho}(w)$ is well-defined and
\begin{equation}\label{tr-sub1}
 w\leq \K_\mu^{\Gw_\rho}[\nu].
\end{equation}

If $\nu_0:=\nu\chr{\bdw}$ then,
\begin{equation}\label{tr-sub2}
\lim_{x\to\partial\Omega}\frac{w(x)}{\K_\mu^{\Omega_\rho}[\nu_0](x)}=1
\quad\text{non-tangentially, $\nu_0$-a.e. on $\partial\Omega$}.
\end{equation}

If $\nu_0=0$ then,
\begin{equation}\label{tr-sub3}
\limsup_{x\to\partial\Omega}\frac{w(x)}{\delta^{\alpha_+}(x)}<\infty.
\end{equation}
 \end{theorem}

\proof The first statement \eqref{tr-sub1} follows from \eqref{KmGm} and Theorem \ref{key-res}~(ii).

 The second statement \eqref{tr-sub2} follows from \eqref{KmGm} and the fact that $\G_\mu^{\Omega_\rho}[\lambda]$ is an $\L_\mu$-potential (i.e. a positive superharmonic function that does not dominate any positive $\L_\mu$-harmonic function). This fact implies (see, e.g. \cite{Ancona-90}):
 $$\lim_{x\to\partial\Omega} \frac{\G_\mu^{\Omega_\rho}[\lambda](x)}{\K_\mu^{\Omega_\rho}[\nu](x)}\to 0 \quad\text{$\nu$-a.e. on $\bdw$}.$$
 By Fatou's limit theorem
 $$\lim_{x\to\partial\Omega} \frac{\K_\mu^{\Omega_\rho}[\nu_0](x)}{\K_\mu^{\Omega_\rho}[\nu](x)}=1
 \quad\text{$\nu$-a.e. on $\bdw$}.$$
Therefore \eqref{KmGm} implies \eqref{tr-sub2}.

The third statement \eqref{tr-sub3} follows from \eqref{tr-sub1} and Corollary \ref{trace=0}.
\qed

\begin{corollary}\label{PL-nu-converse}
Let $w$ be a nonnegative $\L_\mu$-subharmonic function in $\Gw_\rho$ for some $\rho\in (0,\bar\rho)$.
Then $w$ possesses a normalised boundary trace in $\M^+(\partial\Omega)$
if and only if $w$ is dominated by a positive $\L_\mu$-superharmonic function $v$ in a strip around $\bdw$.
\end{corollary}

\proof If $w$ is dominated by a positive $\L_\mu$-superharmonic function in $\Gw_\rho$ then the existence of $\tr_{\bdw} (w)$ follows from \eqref{e-trace-potential} and Theorem \ref{tr-sub}.

Next suppose that $w$ has a normalized boundary trace $\nu_0\in \M^+(\bdw)$. Without loss of generality we may assume that it also has a measure boundary trace $\nu_\rho$ on $\Gs_\rho$. Since $u$ is $\L_\mu$-subharmonic, there exists a positive Radon measure $\tau$ in $\Gw$ \sth
$$-\L_\mu u=-\tau.$$
Let
$\tau_\gb:=\tau\chr{D_\gb\sms \bar D_\rho}$, $w=\K_\mu^{\Gw_\rho}[\nu_0+\nu_\rho]$ and $\nu_\gb=w\lfloor_{\Gs_\gb}$.

 Let $u_\gb$ be the solution of the boundary value problem,
$$\BAL
-\L_\mu v&=-\tau_\gb \text{ in }D_\gb\sms \bar D_\rho,\\
 v&=\nu_\rho \text{ on }\Gs_\rho, \quad v=\nu_\gb \text{ on }\Gs_\gb.
 \EAL$$
Then
$$u_\gb+ \G_\mu^{D_\gb\sms \bar D_\rho}[\tau_\gb]=w.$$
It follows that
$$G_\mu^{\Gw_\rho}[\tau]=\lim_{\gb\to 0}\G_\mu^{D_\gb\sms \bar D_\rho}[\tau_\gb]<\infty,$$
which in turn implies that $\tau \in \M_+(\Gw;\gd^{\ga_+})$ and finally
$$u+ G_\mu^{\Gw_\rho}[\tau]=w.$$
In particular,
\begin{equation}\label{u<Kmr}
u\leq w= \K_\mu^{\Gw_\rho}[\nu_0+\nu_\rho].
\end{equation}
\qed

\begin{corollary}\label{tr=zero}
(i)  Suppose that $u$ is positive and $\L_\mu$-subharmonic  in $\Gw_{\bar\rho}$. Then  $\tr_{\bdw}=0$ if and only if, for every $\rho\in (0,\bar\rho)$, there exists a constant $c_\rho$ such that
\begin{equation}\label{tr*=0}
 u(x)\leq c_\rho\gd(x)^{\ga_+} \forevery x\in \Gw_\rho.
 \end{equation}
(ii) Suppose that $u$ is positive and  $\L_\mu$-superharmonic in $\Gw_{\bar\rho}$. Then $u$ has a normalized boundary trace $\nu\in \M_+(\bdw)$ and consequently there exists $c_\rho$ \sth
\begin{equation}\label{tr*exists}
 \int_{\Gs_\gb}udS\leq c_\rho\gb^{\ga_-} \forevery \gb\in (0,\rho).
 \end{equation}
\end{corollary}

\proof (i)  Obviously \eqref{tr*=0} implies that $\tr_{\bdw}(u)=0$.
Conversely assume that $\tr_{\bdw}(u)=0$.

By the previous corollary $u$ is dominated by an $\L_\mu$-harmonic function. Therefore, by Theorem \ref{tr-sub},
there exist $\lambda\in\M^+_{\delta^{\alpha_+}}(\Omega_\rho)$ and
$\nu\in\M^+(\partial\Omega_{\rho})$ such that
$u=\K_\mu^{\Omega_\rho}[\nu]-\G_\mu^{\Omega_\rho}[\lambda].$
Since $\tr_{\bdw}(u)=0$, $\nu_0=\nu\chr{\bdw}=0$. Hence $u<\K_\mu^{\Omega_\rho}[\nu_\gr]$ where $\nu_\gr=\nu\chr{\Gs_\rho}$. Therefore the result follows from  Corollary \ref{trace=0}.
\medskip

(ii) By the Riesz decomposition theorem (see \cite{Ancona-90}), $u=u_p+u_h$ where $u_p$ is an
$\L_\mu$-potential and $u_h$ is a nonnegative $\L_\mu$-harmonic function in
$\Gw_\rho$. It is known that every $\L_\mu$-potential is the Green potential of a positive measure. Thus there exists $\tau\in \M_+(\Gw;\gd^{\ga_+})$ such that $u_p=\G_\mu^{\Gw_\rho}[\tau]$.
By the Representation Theorem $u_h=\K_\mu^{\Gw_\rho}[\nu]$ for some $\nu\in \M_+(\bdw_\rho)$.  Thus
\begin{equation}\label{u=G+K}
 u=\G_\mu^{\Gw_\rho}[\tau] + \K_\mu^{\Gw_\rho}[\nu].
\end{equation}
The required result follows from Theorem \ref{key-res}.
\qed

\section{$\L_\mu$-moderate solutions of nonlinear equation}\label{sect-3}

In this section we study the nonlinear equation
\begin{equation*}\tag{$P_\mu$}\label{*}
-\L_\mu u+|u|^{q-1}u=0\quad\text{in $\Omega$},
\end{equation*}
where $\Omega\subset\R^N$ is a bounded smooth domain,
$\mu<1/4$ and $q>1$.

\subsection{Preliminaries.}

 Suppose that $u\in L^q_{\loc}(\Gw)$ is either a subsolution or a supersolution of \eqref{*}, in the distribution sense. Then, $u\in W^{1,p}_{\loc}(\Gw)$ for $1\leq p <N/(N-1)$. If, in addition, $u$ is a distributional \emph{solution} of \eqref{*} then it is also a classical solution.

Consequently, if $u\in L^q_{\loc}(\Gw)$ is a distributional subsolution in $\Gw$ then
\begin{equation}\label{weak}
\int_{\Gw}\nabla u \cdot\nabla\varphi\,dx-
\int_{\Gw}\frac{\mu}{\delta^2}u \varphi\,dx\,
+\int_{\Gw}|u|^{q-1}u\varphi\,dx\le 0
\quad\forall\:0\le\varphi\in C^\infty_c(\Gw).
\end{equation}
If, in addition, $u \in H^1_{\loc}(\Gw)$ then \eqref{weak} holds for every $\varphi\in H^1_c(\Gw)$.

A similar statement holds for supersolutions, in which case the inequality sign in \eqref{weak}
is inversed. Of course these statements remain valid for local subsolutions and supersolutions (in a subdomain $G\sbs \Gw$).

We state below two results from \cite{BMR08} that will be used in the sequel.

\begin{lemma} {\sc (Comparison principle \cite{BMR08}*{Lemma 3.2})} \label{comparison}
\begin{enumerate}
\item[$(i)$] Let $G$ be open with $G\subset\Omega$.
Let $0\le \usub, \usup \in H^1_{loc}(G)\cap C(G)$ be a pair of sub and supersolutions to \eqref{*} in $G$
such that $$\limsup_{x\to\partial G}[ \usub(x)-\usup(x)]<0.$$
Then $\usub \leq \usup$ in $G$.
\item[$(ii)$]
Let $G$ be open with $\overline{G}\subset\Omega$.
Let $\usub, \usup\in H^1(G) \cap C(\overline{G})$ be a pair of sub and supersolutions
to \eqref{*} in $G$ and $\usub\leq \usup$ on $\partial G$.
Then $\usub\leq \usup$ in $G$.
\label{hardy_cp}
\end{enumerate}
\end{lemma}

\begin{lemma}\label{l-subsup}
{\sc (\cite{BMR08}*{Lemma 4.10})}
Assume that \eqref{*} admits a subsolution $\usub$ and a supersolution $\usup$ in $\Omega$
so that $0\le\usub\le\usup$ in $\Omega$. Then \eqref{*} has a solution $U$ in $\Omega$
such that $\usub\le U\le\usup$ in $\Omega$.
\end{lemma}

In \cite{BMR08}*{Proposition 3.5} the Keller--Osserman estimate has been extended to equation \eqref{*}. Specifically it was proved that every subsolution $u$ of \eqref{*} in $\Gw$ satisfies,
\begin{equation}\label{KO}
u(x)\le \gamma_\ast\delta^{-\frac{2}{q-1}}(x)\quad\mbox{in }\:\Omega,
\end{equation}
where $\gamma_\ast$ is a constant independent of $u$. In addition it was shown that, if $u$ is a local subsolution  in  $\Gw_\rho$, continuous at $\Gs_\rho$, then $u$ satisfies \eqref{KO} in $\Gw_\rho$, but $\gg_\ast$ may depend on $u$.
We prove below a stronger version that is used later on.

\begin{lemma}{\sc (Keller--Osserman estimate)}\label{OK}
If $u$ is a subsolution of \eqref{*} in $\Gw$ then it satisfies \eqref{KO} with a constant depending only on $q,N,\mu$. If $u$ is a subsolution of \eqref{*} in $\Gw_\rho$ then \eqref{KO} holds with a constant depending only on $q,N,\mu,\rho$ and $\gd(x)$ replaced by
$\gd_\rho(x):=\dist(x,\bdw_\rho)$.
\end{lemma}

\proof Without loss of generality we may assume that $u\geq0$ because $u_+$  is a subsolution. If $\mu\leq0$ then $u$ is also a subsolution of the equation $-\Gd u+ u^q=0$. Therefore in this case \eqref{KO} is a direct consequence of the classical Keller--Osserman inequality.

Now assume that $\mu>0$. Let $y\in \Gw$ and $R=\gd(y)/2$. Then,
$$-\Gd u -\frac{\mu}{R^2} u +u^q\leq 0 \quad\text{in }B_R(y).$$
Therefore in $B_R(y)$ either $u\leq (8\mu/R^2)^{\frac{1}{q-1}}$ or $ -\Gd u + u^q/2\leq 0$.
Hence, by Kato's inequality, the function $v:=(u- (8\mu/R^2)^{\frac{1}{q-1}})_+$ satisfies
$$-\Gd v +v^q/2\leq 0\quad \text{in }B_R(y).$$
By the classical Keller--Osserman inequality,
$$v(y)\leq c(q,N) R^{-\frac{2}{q-1}}.$$
Since $u(y)\leq v(y) + (8\mu/R^2)^{\frac{1}{q-1}}$ we conclude that
\begin{equation}\label{KO1}
u(y)\leq c(\mu,q,N) \gd_\Gw(y)^{-\frac{2}{q-1}} \forevery y\in \Gw.
\end{equation}

Next, let $u$ be a subsolution in $\Gw_\rho$. As before we may assume that $u\geq 0$ and that $\mu>0$.
By the first part of the proof, \eqref{KO1} holds in $\Gw_{3\rho/4}$.  Further,
 $$-\Gd u -(4\mu/\rho^2)u +u^q\leq 0\quad\text{in }\Gw'_\rho=\{x:\, \rho/2\leq\gd(x)<\rho\}.$$
Therefore,
either $u\leq (8\mu/\rho^2)^{\frac{1}{q-1}}$ or $ -\Gd u + u^q/2\leq 0$.
By the same argument as before, the function $v:=(u- (8\mu/\rho^2)^{\frac{1}{q-1}})_+$ satisfies
$$v(x)\leq c(q,N) \dist(x,\Gs_\rho)^{-\frac{2}{q-1}} \forevery x: 3\rho/4\leq \gd(x)<\rho.$$
Consequently,
\begin{equation}\label{KO2}
u(x)\leq c(\mu,q,N,\rho) \dist(x,\bdw_\rho)^{-\frac{2}{q-1}} \forevery x\in \Gw_{\rho}.\qedhere
\end{equation}

\subsection{Moderate solutions.}

We study the generalised boundary trace problem \eqref{*nu}
where $\mu< 1/4$, $q>1$ and $\nu\in \mathcal M^+(\partial\Omega)$.
First we prove,

\begin{lemma}\label{interior-bvp} Let $D$ be a $C^2$ domain \sth $D\Subset \Gw$. If $0\le f\in C(\prt D)$ then there exists a unique solution of the problem
\begin{equation}\label{int-bvp}\left\{\BAL
-\L_\mu u+u^q&=0 \quad\text{in }D,\\
u&=f  \quad\text{on }\prt D.
\EAL\right.
\end{equation}
\end{lemma}

\proof
For $u\in H^1(D)$, let
$$J_D(u)= \int_{D}\big( \frac{1}{2}|\nabla u|^2 - \frac{\mu}{2\gd_{\Omega}^2}u^2 + \frac{1}{q+1}|u|^{q+1}\big)dx.$$
Since $\mu\delta_\Omega^{-2}\in L^\infty(D)$, it is standard to see that
$J_D$ is coercive and weakly l.s.c. on $$H_f^1(D)=\{u\in H^1(D): u=f \;\text{on }\partial D\}.$$
Therefore there exists  a minimizer $u_f\in H^1_f(D)$. We may assume that $u_f>0$ because $|u_f|$ too is a minimizer. The minimizer is a solution of \eqref{int-bvp}. The uniqueness is a consequence of the comparison principle.
\qed
\medskip

Next consider the problem,
\begin{equation*}\tag{$P_{\mu}^{\nu}(\rho)$}\label{*nur}
\left\{
\begin{array}{rcll}
-\L_\mu u+u^q&=&0\quad\text{in $\Gw_\rho$},\smallskip\\
\tr_{\bdw}(u)&=&\nu\chr{\bdw}=:\nu_0,\\
\trm_{\Gs_\rho}(u)&=& \nu\chr{\Gs_\rho}=:\nu_\rho.
\end{array}
\right.
\end{equation*}
where $\mu< 1/4$, $q>1$, $\nu\in \mathcal M^+(\prt \Gw_\rho)$ and $\rho\in (0,\bar\rho]$.

The following result is an adaptation of \cite{MMN}*{Theorem C} to problem \eqref{*nur}. Since $C_H(\Gw_{\bar\rho})=1/4$ the result applies to every  $\mu<1/4$. The proof follows the argument in \cite{MMN}; for the convenience of the reader it is presented below.

\begin{proposition}\label{nl-existence}
Let $\nu\in\M^+(\bdw_\rho)$ and assume that $\K_\mu^{\Omega_\rho}[\nu]\in L^q_{\delta^{\alpha_+}}(\Omega_\rho)$
for some $\rho\in(0,\bar\rho]$.
Then \eqref{*nur}  admits a unique solution $U_\nu$.
\end{proposition}

\proof Let $\{D_n\}$ be a sequence of $C^2$ domains \sth $\bar D_n\sbs D_{n+1}$ and $D_n\uparrow \Gw_\rho$. Let $u_n$ be the solution of \eqref{int-bvp} with $D=D_n$ and $f=f_n:=\K_\mu^{\Gw_\rho}[\nu]\lfloor_{\prt D_n}$. Since $\K_\mu^{\Gw_\rho}[\nu]$ is a supersolution of the equation $\L_\mu v+v^q=0$ in $\Gw_\rho$
it follows that $u_n$ decreases and $u=\lim u_n$ is a solution of this equation. We claim that $u$ is a solution of \eqref{*nur}.
Indeed,
\begin{equation}\label{unDn}
 u_n+\G^{D_n}_\mu[u_n^q]=\BBP^{D_n}_\mu[f_n]=\K_\mu^{\Gw_\rho}[\nu]\quad \text{in }D_n,
\end{equation}
where $\BBP^{D_n}_\mu$ denotes the Poisson kernel of $\L_\mu$ in $D_n$.

Since $u_n\leq \K_\mu^{\Gw_\rho}[\nu]\in L^q_{\gd^{\ga_+}}(\Gw)$ it follows that
$$\G^{D_n}_\mu[u_n^q]\to \G^{\Gw_\rho}_\mu[u^q].$$
Hence, by \eqref{unDn},
$$u+\G^{\Gw_\rho}_\mu[u^q]=\K_\mu^{\Gw_\rho}[\nu] \quad \text{in }\Gw_\rho.$$
By Theorem \ref{key-res}, $\tr_{\bdw} (u)=\nu\chr{\bdw}$ and  (by \eqref{Kest}) $\trm_{\Gs_\rho} (u)=\nu\chr{\Gs_\rho}$.
\qed
\medskip

The next result is  an adaptation of \cite{MMN}*{Theorem D}. We omit the proof which -- except for obvious modifications -- is the same as in \cite{MMN}.

\begin{proposition}\label{nl-asymptotics} Assume that $u$ is a positive solution of \eqref{*nur}. Then
\begin{equation}\label{e-nontang-moderate}
\lim_{x\to\partial\Omega}\frac{u(x)}{\K_\mu^{\Omega_\rho}[\nu_0](x)}=1
\quad\text{non-tangentially, $\nu$-a.e. on $\partial\Omega$},
\end{equation}
where $\nu_0=\nu\chr{\bdw}$.
\end{proposition}

\begin{theorem}\label{Gw-ex}
Let $\nu\in\M^+(\bdw)$ and $\rho\in (0,\bar\rho)$. Let $\nu'\in\M^+(\bdw_\rho)$  be defined by $\nu'=\nu$ on $\bdw$ and $\nu'=0$ on $\Gs_\rho$. Assume that, for some $\rho$ as above, $\K_\mu^{\Omega_\rho}[\nu']\in L^q_{\delta^{\alpha_+}}(\Omega_\rho)$.
Then the boundary value problem \eqref{*nu}  admits a solution in $\Gw$.
\end{theorem}

\proof By Proposition \ref{nl-existence} there exists a (unique) solution $U_{\nu,0}$ of problem $(P_\mu^{\nu'}(\rho))$. For every $k\geq 0$,  let $\nu_k\in \M^+(\bdw_\rho)$ be the measure given by, $\nu_k\chr{\bdw}=\nu$ and $\nu_k\chr{\Gs_\rho}=k dS_{\Gs_\rho}$. By the same proposition there exists a (unique) solution $U_{\nu,k}$ of $(P_\mu^{\nu_k}(\rho))$. Put
$$U_{\nu,\infty}=\lim_{k\to\infty}U_{\nu,k}.$$

Let $R\in (0,\rho)$. By Lemma \ref{interior-bvp} there exists a unique solution $v_R$ of \eqref{int-bvp} in $D_R$ with $f=U_{\nu,0}\lfloor_{\Gs_R}$.
 By the comparison principle,
$$U_{\nu,0}\leq v_R\leq U_{\nu,\infty}\quad \text{in }\Gw_\rho\cap D_R.$$
By Proposition \ref{OK} the family $\{v_R:\,0<R<\rho\}$ is bounded in compact subsets of $\Gw$. Therefore there exists a sequence $\{R_j\}$ converging to zero \sth $v_{R_j}$ converges to a solution $v$ of the nonlinear equation in $\Gw$.
By construction,
$$U_{\nu,0}\leq v\leq U_{\nu,\infty}\quad \text{in }\Gw_\rho.$$
Therefore $\tr_{\bdw}(v)=\nu$.
\qed

\begin{remark}
If $\mu<C_H(\Gw)$ then problem \eqref{*nu} has at most one solution, \cite{MMN}*{Theorem B}. However uniqueness fails  when $C_H(\Gw)<\mu<1/4$. It was proved in \cite{BMR08}*{Theorem 5.3} that in this case there exists a positive solution of \eqref{*nu} with $\nu=0$.
An alternative, more direct proof, is presented in Appendix \ref{App-B}.
\end{remark}

\begin{proposition}\label{mod-sol} Assume that $u\in L^q_{loc}(\Gw)$ is a positive solution of \eqref{*-intro}. Then the following assertions are equivalent:

 (i) $u$  has a normalized boundary trace,

 (ii) $u$ is a moderate solution in the sense of Definition \ref{d-moderate},

 (iii) $u\in L^q(\Gw;\gd^{\ga_+})$.

\end{proposition}

\proof The assumption implies that $\L_\mu u\leq 0$ in $\Gw$. If $\rho\in (0,\bar\rho]$ then, by Lemma \ref{PL-nu-converse}, (i) holds if and only if $u$ is dominated by an $\L_\mu$-superharmonic function in $\Gw_\rho$. Consequently, by Lemma \ref{l-subsup}, (i) holds if and only if $u$ is dominated by an $\L_\mu$-harmonic function in $\Gw_\rho$. Thus (i) and (ii) are equivalent.

If (iii) holds then $v:=u+\BBG_\mu^{\Gw_\rho}[u^q]$ is $\L_\mu$-harmonic. By the representation theorem there exists $\nu\in\M(\bdw_\rho)$ such that $v=\K_\mu^{\Gw_\rho}[\nu]$. Since $\tr_{\bdw}\BBG_\mu^{\Gw_\rho}[u^q]=0$ it follows that $\nu\chr{\bdw}$ is the normalized boundary trace of $u$. Conversely if (ii) holds then by Theorem \ref{tr-sub} $\L_\mu u=u^q\in\M^+_{\delta^{\alpha_+}}(\Omega_\rho)$ which is the same as (iii).
\qed
\subsection{Critical exponents}
The next result provides  necessary and sufficient conditions in order that a positive measures  $\nu\in \M^+(\bdw)$ satisfy,
\begin{equation}\label{K[nu]}
 \K_\mu^{\Omega_\rho}[\nu]\in L^q_{\delta^{\alpha_+}}(\Omega_\rho)\quad\text{for some $\rho>0$}.
\end{equation}

Let $\Gamma_a(x-y)=|x-y|^{-(N-a)}$ denote the Riesz kernel  of order $0<a<N$ in $\R^N$.

\begin{proposition}\label{vg-measure}
Let  $\nu\in \mathcal M^+(\partial\Omega)$.

(i) If $\Gamma_1 *\nu\in L^q_{\gd^{1+(q-1)\ga_-}}(\Gw)$ then $\nu$ satisfies \eqref{K[nu]}.

(ii) Assume $\mu\geq 0$. If $\nu$ satisfies \eqref{K[nu]} then $\BBP_0^\Gw[\nu]\in L^q_{\gd^{1+(q-1)\ga_-}}(\Gw)$.

Here $P_0^\Gw$ is the Poisson kernel of $-\Delta$ in $\Omega$: $P_0^\Gw(x,y)=\gd(x)|x-y|^{-N}$.
\end{proposition}

\proof
 By \eqref{Kest3},
\begin{equation}\label{Kest4}\BAL
 K_\mu^{\Gw_\rho}(x,y)&\sim \frac{\gd(x)^{\ga_+}}{|x-y|^{N-2\ga_-}}
 \sim \gd(x)^{\ga_-}P_0^\Gw(x,y)\big(|x-y|/\gd(x)\big)^{2\ga_-}\\
 &\sim \gd(x)^{\ga_-}\Gg_1(x-y)\big(|x-y|/\gd(x)\big)^{-1+2\ga_-},
 \EAL\end{equation}
for every $(x,y)\in \Gw_{\rho/2}\ti \bdw$.

For every $\mu<1/4$ we have $-1+2\ga_-<0$. Consequently,
\begin{equation}\label{Kest5}
 K_\mu^{\Gw_\rho}(x,y)\leq c \gd(x)^{\ga_-}\Gg_1(x-y) \forevery (x,y)\in \Gw_{\rho/2}\ti \bdw.
\end{equation}
Hence,
$$ \|\K_\mu^{\Omega_\rho}\nu\|_{L^q_{\delta^{\alpha_+}}(\Omega_{\rho/2})}^q \leq
c\int_{\Gw_{\rho/2}}\Big(\int_{\bdw}\Gg_1(x-y)d\nu(y)\Big)^q \gd(x)^{q\ga_- +\ga_+}dx.$$
This proves (i).
\medskip

If $\mu\geq 0$, so that $\ga_-\geq 0$ then, by \eqref{Kest4},
\begin{equation}\label{Kest6}
 K_\mu^{\Gw_\rho}(x,y)\geq c \gd(x)^{\ga_-}P_0^\Gw(x,y) \forevery (x,y)\in \Gw_{\rho/2}\ti \bdw.
\end{equation}
Therefore
$$ \|\K_\mu^{\Omega_\rho}[\nu]\|_{L^q_{\delta^{\alpha_+}}(\Omega_{\rho/2})}^q \geq
c\int_{\Gw_{\rho/2}}\Big(\int_{\bdw}P_0^\Gw(x,y)d\nu(y)\Big)^q \gd(x)^{q\ga_- +\ga_+}dx.$$
This proves (ii).
\qed

Using this result we provide a necessary and sufficient condition for the existence of positive moderate solutions of \eqref{*}.

\begin{proposition}\label{L1data} Let  $\nu\in \mathcal M^+(\partial\Omega)$.

(i) If $\ga_->-\frac{2}{q-1}$ then the boundary value problem \eqref{*nu}
 has a solution for every measure $\nu=f\,dS_{\bdw}$ such that $f\in L^1(\bdw)$.

 (ii) If $\ga_-\leq -\frac{2}{q-1}$ then, for every $\nu \gneq 0$, \eqref{*nu} has no solution.

 \end{proposition}

\noindent\textit{Remark.}\hskip 2mm When $\mu>0$ and consequently $\ga_->0$, the condition in (i) holds for every $q>1$.

\proof Let $\nu=f\,dS_{\bdw}$ and $f\in L^\infty(\bdw)^+$. Let $x\in \Gw_{\gb_0}$ and pick $x'\in \bdw$ such that  $|x'-x|=\gd(x)$. Then,
\begin{equation}\label{int_bdw1}
\BAL \int_{\bdw} |x-y|^{1-N} f(y) dS(y)&\leq c\|f\|_{L^\infty}\Big(\int_{\BSA{l} y\in \bdw\\ |x'-y|\geq \gd(x)\ESA} |x'-y|^{1-N}dS(y) + 1\Big)\\
&\leq c\|f\|_{L^\infty} (1+|\ln \gd(x)|)\leq c' \|f\|_{L^\infty}|\ln \gd(x)|,\EAL
\end{equation}
where $c'$ is independent of $x$.
Therefore, if $(q-1)\ga_- +1>-1$ then $\Gg_1*\nu\in L^q_{\gd^{1+(q-1)\ga_-}}(\Gw)$. Consequently, by Proposition \ref{vg-measure}~(i) and Theorem \ref{Gw-ex}, problem \eqref{*nu} has a solution.

Next, let $f\in L^1(\bdw)^+$ and $\nu=f\,dS_{\bdw}$. If $\nu_n=\min(f,n)dS_{\bdw}$ then problem $(P_\mu^{\nu_n})$
has a solution $u_n$ and the sequence $\{u_n\}$ is non-decreasing. In view of the Keller--Osserman estimate \eqref{KO},
$\{u_n\}$ converges to a solution $u$ of \eqref{*nu}. This proves (i).

We turn to part (ii). Suppose that $\ga_-\leq -\frac{2}{q-1}$ and that there exits $\nu\in \mathcal M^+(\bdw)\sms \{0\}$ such that problem \eqref{*nu} has a solution $u$. Then, there exists $c>0$ \sth
$$c\gb^{-\frac{2}{q-1}}\leq c\gb^{\ga_-}\leq \int_{\Gs_\gb} \K_\mu^{\Gw_\rho}[\nu] dS \forevery \gb\in (0,\gb_0).$$

Since $u=-\G_\mu[u^q]+\K_\mu[\nu]$ and $\tr_{\bdw}(\G_\mu[u^q])=0$ it follows that, for sufficiently small $\gb_1$,
\begin{equation}\label{u-growth}
c\gb^{\ga_-}\leq \int_{\Gs_\gb} u dS \forevery \gb\in (0,\gb_1).
\end{equation}
But, by the Keller-Osserman estimate, $u(x)\leq c_1\gd(x)^{-\frac{2}{q-1}}$ so that \begin{equation}\label{u-growth+}
c\gb^{\ga_-}\leq \int_{\Gs_\gb} u dS \leq c_2 \gb^{-\frac{2}{q-1}}\forevery \gb\in (0,\gb_1).
\end{equation}
If $\ga_-<-2/(q-1)$ we reached a contradiction. If $\ga_- = -2/(q-1)$ then, in view of the Keller-Osserman estimate \eqref{KO} we conclude that $u(x)\sim \gd(x)^{-\frac{2}{q-1}}$. This implies that $u\sim U_{max}$ ($=$ the maximal solution of $-\L_\mu v+v^q=0$). Thus $\sup U_{max}/u:=c<\infty$. Now $cu$ is a supersolution and, if $v$ is the largest solution dominated by $cu$ then $\tr(v)=c\,\tr(u)=c\nu$. It follows that $U_{max}\leq v$ which is impossible.
\qed
\medskip

\begin{remark} When $\mu>0$ -- and consequently $\ga_->0$ -- the condition in (i) holds trivially for every $q>1$.
However, if $\mu<0$  and
$$q\geq q_{\mu}^*:=1-\frac{2}{\alpha_-}$$
then equation \eqref{*-intro} has no moderate solution except for the trivial solution.
\end{remark}

\begin{lemma}\label{qcmu}
Let $\mu<C_H(\Gw)$ and put
$$q_{\mu,c}=\frac{N+1-\alpha_-}{N-1-\alpha_-}.$$
Then, for $y\in \bdw$,
$$ K_\mu^\Gw(\cdot,y)\in {L^q(\Gw, \gd^{\ga_+})} \Longleftrightarrow q<q_{\mu,c}.$$
For every $q\in(1,q_{\mu,c})$  there exists a number $c=c(q,N,\mu)$ such that
 \begin{equation}\label{Knu}
\|K_\mu^{\Omega}[\nu]\|_{L^{\frac{N+\alpha_+}{N-1-\alpha_-}}(\Omega, \delta^{\alpha_+})}\leq c\|\nu\| \forevery \nu\in \M(\bdw).
 \end{equation}
\end{lemma}

\proof
Recall that
\begin{equation}\label{Kmu}
K_\mu^\Gw(x,y)\sim |x-y|^{2-N-\ga_+}(\gd(x)/|x-y|)^{\ga_+}=\gd(x)^{\ga_+}|x-y|^{2\ga_- -N},
\end{equation}
(see \cite{MMN}*{Section 2.2}).
Therefore,
$$c'(\frac{\gd(x)}{|x-y|})^{\ga_+}|x-y|^{1+\ga_- -N}\leq K_\mu(x,y)\leq c|x-y|^{1+\ga_- -N}.$$
It follows that $K_\mu(\cdot,y)\in {L^q(\Gw, \gd^{\ga_+})}$ if and only if
$$I:=\int_0^1 t^{q(1+\ga_- -N)}t^{\ga_+}t^{N-1}dt<\infty$$
and
$$\|K_\mu(\cdot,y)\|_{L^q(\Gw, \gd^{\ga_+})} \sim I.$$
A simple computation shows that $I<\infty$ if and only if
$$q<q_{\mu,c}=\frac{N+1-\ga_-}{N - 1-\ga_-}.$$

Finally,
$$\|K_\mu^{\Omega}[\nu]\|_{L^q(\Omega, \delta^{\alpha_+})}\leq \int_{\bdw}\|K_\mu(\cdot,y)\|_{L^q(\Gw, \gd^{\ga_+})}d|\nu|(y)\leq c\|\nu\|.$$
\qed

\begin{corollary} Let $\mu<1/4$. If $1<q<q_{\mu,c}$
then the boundary value problem \eqref{*nu}
 has a solution for every Borel measure $\nu$. Moreover, if $q\geq q_{\mu,c}$ then problem \eqref{*nu} has no solution when $\nu$ is the Dirac measure.
\end{corollary}

\proof
In view of Lemma \ref{qcmu}, the first assertion follows from Theorem \ref{Gw-ex}. The second assertion follows from Proposition \ref{nl-asymptotics}.
\qed

\appendix

\section{Non-uniqueness for $C_H(\Omega)<\mu< 1/4$}\label{App-B}

We are going to show that for $C_H(\Omega)<\mu< 1/4$ the problem
\begin{equation*}\tag{$P_{\mu}^0$}\label{*nu-0}
\left\{
\begin{array}{rcll}
-\L_\mu u+u^q&=&0&\text{in $\Gw$},\smallskip\\
\tr_\mu(u)&=&0,&
\end{array}
\right.
\end{equation*}
admits a nontrivial solution. This was proved in \cite{BMR08}*{Theorem 5.3}. Here we provide a more direct argument.

Recall that if $C_H(\Omega)< 1/4$ then the operator $-\L_{C_H(\Omega)}$ admits a positive ground state solution
$\phi_H\in H^1_0(\Omega)$ such that $-\L_{C_H(\Omega)}\phi_H=0$ in $\Omega$, see \cite{MMP}.

\begin{proposition}\label{p-nu0}
Assume that $C_H(\Omega)<\mu< 1/4$ and $q>1$.
Then \eqref{*nu-0} admits a positive solution $U_0$ such that
$$\liminf_{x\to\partial\Omega}\frac{U_0(x)}{\phi_H(x)}>0.$$
\end{proposition}

\proof
Since $-\L_{C_H(\Omega)}\phi_H=0$ in $\Omega$, for a small $\tau>0$ we obtain
$$-\L_\mu(\tau\phi_H)+(\tau\phi_H)^q=-\frac{\mu-C_H(\Omega)}{\delta^2}(\tau\phi_H)+(\tau\phi_H)^q\le 0\quad\text{in $\Omega$},$$
so that $\tau\phi_H$ is a subsolution for \eqref{*nu-0} in $\Omega$.

Fix $\rho\in(0,\bar\rho]$.
Similarly to the proof of Theorem \ref{Gw-ex},
for every $k\geq 0$ denote $\nu_{\rho,k}=k dS_{\Gs_\rho}$ and let $\nu\in \M^+(\bdw_\rho)$ be the measure \sth $\nu\chr{\bdw}=0$ and $\nu\chr{\Gs_\rho}=\nu_{\rho,k}$. By Proposition \ref{nl-existence} there exists a (unique) solution of \eqref{*nur} with this boundary data. Denote this solution by $U_{0,k}$ and put
$$U_{0,\infty}=\lim_{k\to\infty}U_{0,k}.$$
Let $R\in (0,\rho)$. By Lemma \ref{interior-bvp} there exists a unique solution $v_R$ of \eqref{int-bvp} in $D_R$ with $f=2U_{0,\infty}$ on $\Gs_R$.
We define,
\begin{equation*}
\usup:=\min\{U_{0,\infty},u_R\}\quad \text{in } D_R\cap\Gw_\rho.
\end{equation*}
Then $\usup$ is a supersolution of $(P_\mu)$ in $D_R\cap\Gw_\rho$, $\usup=U_{0,\infty}$ in $D_R\cap\Gw_{\rho'}$ for some $\rho'\in (R,\rho)$ and $\usup=u_R$ in $D_{R'}\cap\Gw_\rho$ for some $R'\in (R,\rho')$. Therefore  setting $\usup=u_R$ in $\Gw\sms\Gw_\rho$ and $\usup=U_{0,\infty}$
in $\Gw\sms D_R$ provides an extension (still denoted by $\usup$) that is a supersolution of $(P_\mu)$ in $\Gw$. As $\usup= U_{0,\infty}$ in a neighborhood of $\bdw$ it follows that $\usup\sim \gd^{\ga_+}$ in such a neighborhood. On the other hand $\phi_H\sim \gd^{a_+}$ where $a_+:=\frac{1}{2}+\sqrt{\frac{1}{4}-C_H(\Gw)}$. As $C_H(\Gw)<\mu$ it follows that $\ga_+<a_+$ so that $\gd^{\ga_+}>\gd^{a_+}$. Therefore $\tau\phi_H<\usup$ near $\bdw$ and therefore, by Lemma \ref{comparison},
everywhere in $\Gw$. Finally by  Lemma \ref{l-subsup} we conclude that there exists a solution $U_0$ of $(P_\mu)$ in $\Gw$ such that $\tau\phi_H<U_0<\usup$. Thus $U_0$ is a positive solution such that $\tr (U_0)=0$.
\qed

{\small
\subsection*{Acknowledgements}
The first author wishes to acknowledge the support of the Israel Science Foundation, funded by the Israel Academy of Sciences and Humanities, through grant 91/10.
Part of this work was carried out during the visit of the second author to Technion, supported by the Joan and Reginald Coleman-Cohen Fund.
VM is grateful to Coleman-Cohen Fund and Technion for their support and hospitality.

The authors wish to thank Catherine Bandle and Yehuda Pinchover for many fruitful discussions.
}

\end{document}